\documentclass[12pt,a4paper]{article}
\usepackage[utf8]{inputenc}
\usepackage[T1]{fontenc}
\usepackage{amsmath}
\usepackage{amssymb}
\usepackage{amsthm}
\usepackage{graphicx}
\usepackage{mathrsfs}
\usepackage{float}

\newtheorem{example}{Example}[section]

\newtheorem{theorem}{Theorem}[section]
\newtheorem{lemma}{Lemma}[section]

\title{ Finite groups whose commuting graphs are line graphs}	 
	\author {Siddharth Malviy\footnote{First author, Email:  malviysiddharth@gmail.com} ~and Vipul Kakkar\footnote{Corresponding author, Email: vplkakkar@gmail.com }\\ Department of Mathematics \\
		Central University of Rajasthan \\Ajmer, India}

\date{}
\begin{document}
	\maketitle
	\noindent \textbf{{Abstract.}}
	The commuting graph ${\Gamma(G)}$ of a group $G$ is the simple undirected graph with group elements as a vertex set and two elements $x$ and $y$ are adjacent if and only if $xy=yx$ in $G$. By eliminating the identity element of $G$ and all the dominant vertices of $\Gamma(G)$, the resulting subgraphs of $\Gamma(G)$ are $\Gamma^*(G)$ and $\Gamma^{**}(G)$, respectively. In this paper, we classify all the finite groups $G$ such that the graph  $\Delta(G) \in \{\Gamma(G), \Gamma^*(G), \Gamma^{**}(G)\}$ is the line graph of some graph. We also classify all the finite groups $G$ whose graph  $\Delta(G) \in \{\Gamma(G), \Gamma^*(G), \Gamma^{**}(G)\}$ is the complement of line graph. \\
	
	
	
	\section{Introduction}
	
	In the last few decades the interest of people in the study of algebraic objects
	using graph theoretic concepts is growing which is an interesting research
	topic leading to several important results and questions. The study of graphs over many
	algebraic structures is very important as graphs of this type have numerous
	applications (\cite{cjo}, \cite{sat}). The commuting graphs on a group $G$ were introduced by Brauer and Fowler \cite{brf} with vertex set $G\setminus\{e\}$.
	The commuting graphs for different non-abelian groups have been studied by many authors (see \cite{bd}, \cite{fms1} \cite{mmpp}). The graph theoretic properties such as distance detour, metric dimension and resolving polynomial 
	properties of the commuting graph on the dihedral group $D_{n}$ were studied by
	Faisal \textit{et al.} \cite{fms1}. Authors in \cite{sml}, also studied the detour distance properties, resolving polynomial and spectral properties of the commuting graph of non-abelian groups of order $p^4$ with center having $p$ elements. Recently, Carleton et al. \cite{solvable}, studied the commuting graph for A-solvable groups and Ashrafi et al. \cite{Ashrafi}, studied the commuting graph of CA-groups. 
	
	The line graph $L(\Gamma)$ of graph $\Gamma$ is the graph whose vertex set consists of all edges of $\Gamma$; two vertices of $L(\Gamma)$ are adjacent if and only if they are incident in $\Gamma$.
	All finite nilpotent groups whose power graphs and proper power graphs are line graphs were characterized by Bera \cite{bera}. Parveen et al. characterized all finite groups whose enhanced power graphs are line graphs in \cite{jk and dalal}. Furthermore, \cite{jk and dalal} determines all finite nilpotent groups whose proper enhanced power graphs are line graphs of certain graphs. In \cite{manisha}, Manisha et al. characterized all the finite groups whose order supergraph is the line graph. Throughout this paper, $G$ is a finite group and $e$ is the identity element of $G$.\\
	In this paper, we aim to study the line graphs of commuting graph associated to finite groups. By eliminating the identity element of $G$ and all the dominant vertices of $\Gamma(G)$, the resulting subgraphs of $\Gamma(G)$ are $\Gamma^*(G)$ and $\Gamma^{**}(G)$, respectively. We characterize all the finite group $G$ such that $ \Delta(G) \in \{\Gamma(G), \Gamma^*(G), \Gamma^{**}(G) \} $ is a line graph of some graph. Also, we classify all finite groups $G$ such that $\Delta(G) \in \{\Gamma(G), \Gamma^*(G), \Gamma^{**}(G) \} $ is the complement of a line graph.
	
	\section{Preliminaries}
	The vertex set $V(\Gamma)$ and the edge set $E(\Gamma)\subseteq V(\Gamma)\times V(\Gamma)$ form an ordered pair that constitutes a graph $\Gamma$. If $\{u, v\}\in E(\Gamma)$, then two vertices, $u$ and $v$ are adjacent; if so, we denote them as $u \sim v$ and if not, $u \nsim v$. When a pair of edges $e_1$ and $e_2$ have a similar endpoint, then they are referred to as {incident edges}. If a graph has no loops or multiple edges, it is referred to as a {simple graph}. In this study, we just take into consideration simple graphs.  A graph $\Gamma'$ such that $V(\Gamma')\subseteq V(\Gamma)$ and $E(\Gamma')\subseteq E(\Gamma)$ is called a subgraph  of a graph $\Gamma$.  
	
	Suppose that $X\subseteq V(\Gamma)$. Then the subgraph $\Gamma'$ induced by the set $X$ is a graph such that $V(\Gamma')=X$ and $u,v\in X$ are adjacent if and only if they are adjacent in $\Gamma$. A vertex $u$ of a graph $\Gamma$ is referred to as a dominating vertex of $\Gamma$ if it is adjacent to every other vertex of $\Gamma$. We refer to the set of all dominating vertices of $\Gamma$ as $\mathrm{Dom}(\Gamma)$. A graph $\Gamma$ is considered complete if every pair of vertices is adjacent to one another. $K_n$ represents a complete graph with $n$ vertices.  The graph $\overline{\Gamma}$ such that $V(\Gamma)= V(\overline{\Gamma})$ and two vertices $u$ and $v$ are adjacent in $\overline{\Gamma}$ if and only if $u$ is not adjacent to $v$ in $\Gamma$ is the complement of a graph $\Gamma$. 
   
	   Throughout this paper, $\mathbb{Z}_n, D_n, S_n, A_n$ and $Q_8$ denotes the cyclic group of order $n$, dihedral group of order $2n$, symmetric group on $n$ symbols, alternating group on $n$ symbols and the quaternion group of order $8$ respectively. The centralizer of an element $x$ in the group $G$ is denoted by $C_G(x)$ and the center of the group $G$ is denoted by $Z(G)$.

	\noindent A characterization of line graph and its complement are described in the next two lemmas, both of which are helpful in the sequel.
	
	\begin{lemma}{\rm \cite{line}}{\label{induced lemma}}
		A graph $\Gamma$ is the line graph of some graph if and only if none of the nine graphs in $\mathrm{Figure \; \ref{figure 1}}$  is an induced subgraph of $\Gamma$.
	\end{lemma}
	\begin{figure}[H]
		\centering
		\includegraphics[scale=.7]{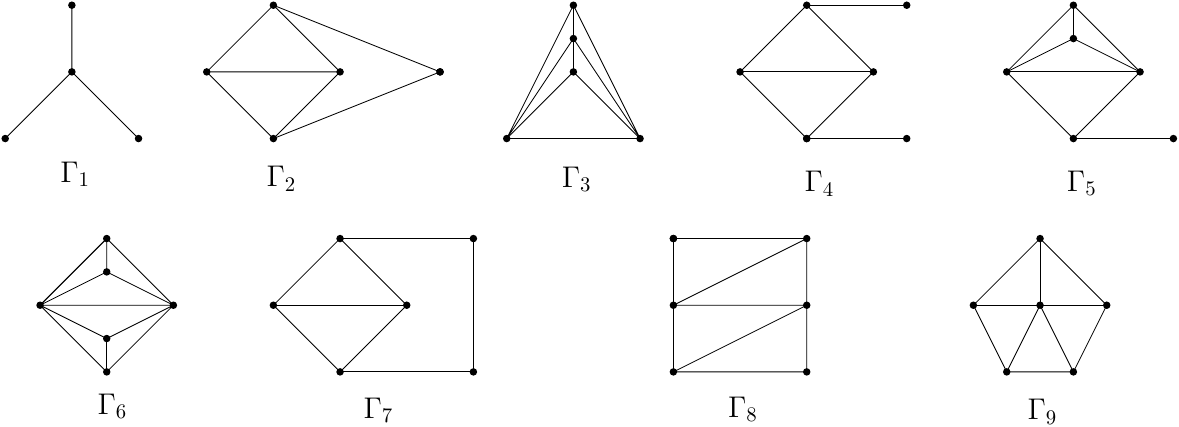}
		\caption{Forbidden induced subgraphs of line graphs.}
		\label{figure 1}
	\end{figure}
	
	\begin{lemma}{\rm \cite[Theorem 3.1]{line complement}} {\label{complement induced lemma}}
		A graph $\Gamma$ is the complement of a line graph if and only if none of the nine graphs $\overline{\Gamma_i}$ in $\mathrm{Figure \; \ref{figure 2}}$ is an induced subgraph of $\Gamma$.
		\begin{figure}[H]
			\centering
			\includegraphics[scale=.7]{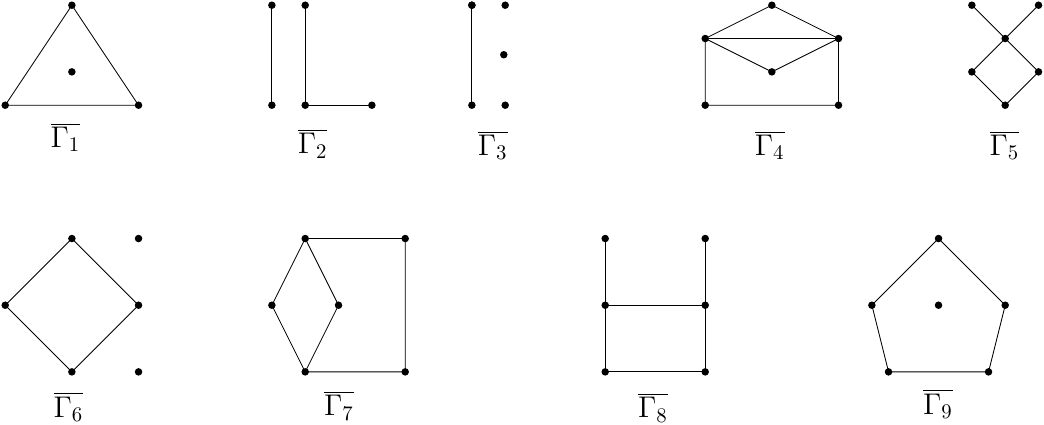}
			\caption{Forbidden induced subgraphs of the complement of line graphs.}
			\label{figure 2}
		\end{figure}
	\end{lemma}

	\section{Line graph characterization of $\Gamma(G)$} 
	
	All the finite groups $G$ such that $\Gamma(G)$ is a line graph of some graphs are classified in this section. Afterwards, we identify all the finite groups that have $\Gamma^*(G)$ and $\Gamma^{**}(G)$ as line graphs. Lastly, we characterize all the groups $G$ that have $\Gamma(G)$, $\Gamma^*(G)$, and $\Gamma^{**}(G)$ as the complement of the line graph of some graph. One can easily observe the following.
	
	\begin{lemma}
		If a graph is a complete graph, then it is the line graph and complement of line graph of some graph.
	\end{lemma}

\begin{lemma}
	The commuting graph $\Gamma(G)$ of a group $G$ is complete if and only if $G$ is abelian.
\end{lemma}

\begin{lemma}
	For the commuting graph $\Gamma(G)$, the dominating set $Dom(\Gamma)$ is the center $Z(G)$ of a group $G$.
\end{lemma}

\begin{lemma} \label{mantel lemma}
\cite{Mantel}	The maximum number of edges in an	$n$-vertex triangle-free graph is $\lfloor$ $\frac{n^2}{4} \rfloor.$
\end{lemma}
	
	\begin{theorem}\label{theorem 3.1}
		The commuting graph $\Gamma(G)$ is the line graph of some graph if and only if $G$ is abelian.
	\end{theorem}

\begin{proof}
If $G$ is an abelian group then $\Gamma(G)$ is the complete graph. Hence it is a line graph of some graph. Now, we show that no non-abelian group can be a line graph.
Let $G$ be a non-abelian group. \\
 Let $\mid Z(G) \mid \geq 3.$ Let three distinct elements in the center be $e,x$ and $y$. Since $G$ is a non-abelian group, there exist elements $a$ and $b$ such that $a \nsim b$. Then the set $\{e,x,y,a,b\}$ will make the structure of $\Gamma_3$ in Figure \ref{figure 1}. Therefore, $\mid Z(G) \mid \leq 2.$
Now, suppose that $Z(G)=\{e,x\}.$  Note that there exist $a$ and $b$ such that $a \nsim b$. Note that $a \sim ax$ and  $b \sim bx$.  The set $\{e,x,a, ax,b, bx\}$ will make the structure of $\Gamma_6$ in Figure \ref{figure 1}. Therefore, the group $G$ has the trivial center.\\ Note that the probability of any two elements in $G$ to commute is $$P_2(G)=\frac{\text{Nubmer of conjugacy classes in $G$}}{\text{Total number of elements in $G$}} (\mbox{see \cite{probability}}).$$
One can easily note that if a group $G$ has trivial center, then $P_2(G) \leq \frac{1}{2}.$ 
Let $P_2(G)=\frac{1}{2}$. Then the maximum of half pairs of elements can commute. There is total $^n C_2$ pairs of $n$ elements in which $\frac{n(n-1)}{4}$ pairs can commute where $n= \mid G \mid$. This implies that there exist  $\frac{n(n-1)}{4}$ pairs that do not commute.\\
Let $\overline{\Gamma(G)}$ be the complement of the commuting graph $\Gamma(G)$. By Lemma \ref{mantel lemma}, the maximum number of edges in $(n-1)$ non-central vertex triangle free graph in  $\overline{\Gamma(G)}$ is $\lfloor$ $\frac{(n-1)^2}{4} \rfloor.$ Note that there are $\frac{n(n-1)}{4}$ pairs which do not commute and  
\begin{equation} \label{eq. 1}
	\frac{n(n-1)}{4} >  \lfloor \frac{(n-1)^2}{4} \rfloor.
\end{equation}
Hence, there always exists three distinct elements $x,y$ and $z$ such that $x$ does not commte with $y$, $y$ does not commte with $z$ and $z$ does not commte with $x$. As, the identity element always commute with all other elements, so the set $\{e,x,y,z\}$ will make $\Gamma_1$ in Figure \ref{figure 1}.\\
If $P_2(G) < \frac{1}{2}$, then there are more choices of pairs which do not commute. One can easily check that the inequality (\ref{eq. 1}) is satisfied in this case. Therefore, we get an induced subgraph $\Gamma_1$ of Figure \ref{figure 1} in this case. Therefore, there does not exist a non-abelian group for that the commuting graph $\Gamma(G)$  is the line graph of some graph. 
 \end{proof}

\begin{theorem} \label{theorem 3.2}
		The commuting graph $\Gamma^*(G)$ is the line graph of some graph if and only if one of the condition holds:
		\begin{enumerate}
			\item The group $G$ is abelian.
			\item The group $G$ is a non-abelian group with trivial center and  the centralizer of any non-central element is abelian.
		\end{enumerate}
\end{theorem}

\begin{proof}
	If $G$ is abelian group, then $\Gamma^*(G)$ is the complete graph. Hence it is a line graph of some graph.
	Now, let $G$ be a non-abelian group.\\
	By the similar argument as in the proof of Theorem \ref{theorem 3.1}, $\mid Z(G) \mid \leq 2$. Now, suppose $Z(G)=\{e,x\}$. Since the conjugacy class of each non-central element contains more than one element, one can observe that $P_2(G) > \frac{1}{2}$ if and only if $G$ is either isomorprhic to $D_4$ or  $Q_8$, but $\Gamma^*(D_4)$ and $\Gamma^*(Q_8)$ are not line graphs. If $P_2(G) \leq \frac{1}{2}$, then by the similar argument as in Theorem \ref{theorem 3.1}, we get three distinct element $y,z$ and $w$ such that $ y\nsim z, z\nsim w$ and $w \nsim y$. Now the set $\{x,y,z,w\}$ will make $\Gamma_1$ of Figure \ref{figure 1}. Therefore, the group $G$ has the trivial center.\\
	Let $x \in G$ be a non-central element. Suppose that the centralizer $C_G(x)=\{e,x,y_1, \hdots, y_{m-2}\}$ of $x$ in $G$ is non-abelian. By the similar argument as in Theorem \ref{theorem 3.1}, we can suppose $\mid Z(C_G(x))\mid =2.$ By the similar argument for $C_G(x)$ as above, one can show that $\Gamma^*(G)$ is not a line graph. This implies that the centralizer of each non-central element is abelian. One can easily observe that $\Gamma^*(G)$ is a line graph in this case. 
\end{proof}

\begin{example}
	The commuting graph $\Gamma^*(D_n)$ is the line graph for dihedral groups $D_n$ when $n$ is odd.
\end{example}

If $G$ is abelian group, then the center consists all the elements of group. In this case there is no vertex left for $\Gamma^{**}(G)$. So, we will find when the commuting graph $\Gamma^{**}(G)$ is the line graph of some graph for non-abelian groups.

\begin{theorem}
		Let $G$ be a non-abelian group. Then the commuting graph $\Gamma^{**}(G)$ is the line graph of some graph if and only if centralizer of any non-central element is abelian.
\end{theorem}

\begin{proof}
 If the centralizer of each element is abelian, then $G$ is partitioned in commuting classes. One can easily check that $\Gamma^{**}(G)$ is a line graph in this case.\\
If the centralizer of an element in a group $G$ is non-abelian, then by the similar argument as in Theorem \ref{theorem 3.2}, one can show that  $\Gamma^{**}(G)$ is not a line graph.
\end{proof}

\begin{example}
	The commuting graph $\Gamma^{**}(D_n)$ is the line graph for dihedral groups $D_n$.
\end{example}

\begin{example}
	The commuting graph $\Gamma^{**}(Q_8)$ is the line graph for quaternion groups $Q_8$.
\end{example} 

\begin{theorem}
		Let $\Delta(G) \in \{ \Gamma(G), \Gamma^*(G), \Gamma^{**}(G) \}.$ Then $\Delta(G)$ is the complement of the line graph of some graph if and only if one of the condition holds:
		\begin{enumerate}
			\item The group $G$ is abelian.
			\item If $G$ is a non-abelian group, then $G \cong D_4$ or $G \cong Q_8$.
		\end{enumerate}
\end{theorem}

\begin{proof}
		If $G$ is abelian group, then $\Delta(G)$ is the complete graph. Hence it is the complement of a line graph of some graph. Now, let $G$ be a non-abelian group. \\ 
	    Let $\mid Z(G)\mid > 2.$ Then there exists two elements in center besides the identity. Let $x$ and $y$ to be two distinct elements of center other than the identity. As $G$ is non-abelian group, there exists $a$ and $b$ in $G$ such that $ a \nsim b$. Now, $a, ax$ and $ay$ commute with each other. But $b$ does not commute with all three of them. So the set $\{a,ax,ay,b\}$ will make $\overline{\Gamma_1}$ of Figure \ref{figure 2}. Hence, for the complement of line graph, $\mid Z(G) \mid \leq 2.$\\
	    Let $Z(G)=\{e,x\}$. First, suppose that the centralizer of every non-central element $a$ as $\{e,x,a,ax\}.$ Then the set of non-central elements is partitioned into the commuting classes consists of only two elements. If the number of commuting classes of non-central elements is greater than three, then we get $\overline{\Gamma_3}$ of Figure \ref{figure 2}. This implies that there are three or less commuting classes. The only possibility for such group are either $D_4$ or $Q_8$. One can easily check that $\Delta(D_4)$ and $\Delta(Q_8)$ are complement of a line graph.\\
	    Now, suppose that the centralizer $C_G(a)$ of every non-central element $a \in G$ is abelian and some centralizer $C_G(a)$ consists of more than four elements. This implies that $C_G(a)\backslash Z(G)$ consists of atleast three distinct elements $a_1, a_2$ and $a_3$. Since, $G$ is non-abelian, there exists $y$ in $G\backslash C_G(a)$ which does not commute with any of $a_1, a_2$ and $a_3$. Therefore, we get an induced subgraph $\overline{\Gamma_1}$ of Figure \ref{figure 2}. \\
	    Now, suppose that the centralizer of $C_G(a)$ of a non-central element $a \in G$ is non-abelian. Then the centralizer $C_G(a)$ consists of atleast two elements $c$ and $d$ such that $ c \nsim d.$ Note that $c \sim cx, cx \sim ca, c \sim ca, c \nsim d, d \nsim cx$ and  $ d \nsim ca$. Therefore, we get an induced subgraph $\overline{\Gamma_1}$ of Figure \ref{figure 2}. Hence, $\mid Z(G) \mid =1.$\\
	    Now, suppose an odd prime $p \geq 5$ divide the order $\mid G \mid $ of $G$. Then there exists an element $x \in G$ such that $ o(x)=p$. Therefore, we get a cyclic subgroup 	    $$H= \langle x \rangle= \{e,x,x^2, \hdots x^{p-1}\} \cong \mathbb{Z}_p.$$ 
	    Note that $C_G(x) \subseteq C_G(x^2) \subseteq C_G(x^4).$ Since $G$ is non-abelian, we get an element $y \in G \backslash C_G(x^4)$ such that $y \nsim x, y \nsim x^2, y \nsim x^4$. Therefore, we get an induced subgraph $\overline{\Gamma_1}$ of Figure \ref{figure 2}. This implies that only divisors of the order $\mid G \mid $ of $G$ are $2$ or $3$ such that the order of each element is less than five. Therefore, $\mid G \mid= 2^\alpha 3^\beta, \alpha, \beta \geq 1.$ Then, there exists a subgroup $H$ of $G$ generated by an involution and an element of order $3$. Hence,  by \cite{2-3 generated}, $G$ is isomorphic to either $A_4, S_3$ or $S_4$. One can easily check that $\Delta(A_4), \Delta(S_3)$ and $\Delta(S_4)$ are not complement of a line graph.		
\end{proof}
	
	\section*{Declarations}
	
	\noindent \textbf{Acknowledgement}: The first author is supported by junior research fellowship of CSIR, India.\\
	
	\noindent\textbf{Conflicts of interest}: There is no conflict of interest regarding the publishing of this paper.

\end{document}